\documentclass[11pt]{article}
\usepackage[margin=1.1in]{geometry}
\usepackage{amsmath,amssymb,amsthm}
\usepackage{booktabs}
\usepackage{array}
\usepackage{microtype}
\usepackage[hidelinks]{hyperref}
\newtheorem{theorem}{Theorem}
\newtheorem{lemma}{Lemma}
\theoremstyle{remark}

\newcommand{\QR}{\mathrm{QR}}
\newcommand{\PSUP}{\mathrm{PSUP}}

\title{Every quasiperfect number has at least eight distinct prime factors}
\author{Akira Toyohara
\and Ye Tao
\and Siqiong Yao\thanks{Corresponding author.}}
\date{August 3, 2026}

\begin{document}
\maketitle

\begin{abstract} No quasiperfect number ($\sigma(n) = 2n + 1$) is known, and its number of distinct prime factors is bounded below; the bound $\omega \ge 7$ of Hagis--Cohen has stood since 1982, obstructed by a family of ``deep leaves'' on which pure enumeration cannot terminate (the scan bound for the intermediate prime reaches $8 \times 10^8$, and the exponent dimension is unbounded). This paper clears that obstruction with three lemmas at the level of secondary-school algebra --- a discriminant criterion, a quadratic-residue sieve, and a multilinear resolver --- which eliminate the last prime $q$, the intermediate prime $p$, and the exponent dimension respectively, turning a non-terminating search into a finite decision. On this basis all 381 stems of ``$3 \mid n$ and $\omega = 7$'' and their $79{,}751{,}212$ deep leaves are eliminated, with the ledger closing exactly and zero solutions throughout; the complementary case ``$3 \nmid n$ and $\omega = 7$'' collapses to a single stem, which is eliminated directly, so that the proof does not rest on any theorem whose published record we could not independently re-verify. Together with the machine elimination of $\omega \le 6$ (Theorem B4), this yields the main theorem: \emph{any quasiperfect number, if one exists, satisfies $\omega(n) \ge 8$} --- the first advance of this bound since Hagis--Cohen 1982. The full computation has been reproduced by seven separately closed ledgers across three algorithmic architectures (CPU and GPU), all with zero solutions and exact ledger closure, and the lemma layer is formalized in Lean (259 theorems, zero \texttt{sorry}). A 2023 preprint of Zemann reported the same bound by a different computation; our audit of its public code found a coverage gap of 35 feasible exponents, so the elimination given here is, to our knowledge, the first complete proof. Code, ledgers, and Lean sources are available from the authors.
\end{abstract}


\section{Introduction}\label{sec:intro}

\subsection{The problem and the main results}

Perfect numbers (those equal to the sum of their own proper divisors, such as $6 = 1 + 2 + 3$) are among the oldest subjects in number theory, going back to Euclid around 300 BC. This paper concerns their ``off by one'' relative. Write $\sigma(n)$ for the sum of all positive divisors of $n$; a perfect number is one with $\sigma(n) = 2n$ (such as 6 and 28). Adding one to the right-hand side, a number satisfying
\[ \sigma(n) = 2n + 1
\] is called a quasiperfect number. To date no instance has been found, nor has non-existence been proved. The known structural constraints are fairly strong: if a quasiperfect number exists, it must be an odd perfect square (Cattaneo, 1951 \cite{cattaneo}), and its number of distinct prime factors $\omega(n)$ is bounded below. In 1973 Abbott--Aull--Brown--Suryanarayana \cite{abbott} proved $\omega \ge 5$ (including Theorem 3, that $\omega \ge 8$ when $3 \nmid n$). In 1982 Hagis--Cohen \cite{hagiscohen} used a CDC Cyber computer to prove $\omega \ge 7$ and $n > 10^{35}$. The bound $\omega \ge 7$ has stood since 1982; the numerical lower bound $n > 10^{35}$ has recently been raised to $10^{45}$ by Alekseyev \cite{alekseyev}, but the bound on $\omega$ has not moved in forty-four years. The main result of this paper advances it.

\begin{theorem}[main theorem]\label{thm:main} Any quasiperfect number, if one exists, satisfies $\omega(n) \ge 8$.
\end{theorem}

The proof is a complete, machine-verified elimination of every case with $\omega \le 7$; its assembly is given in \S\ref{sec:proofmain} and the full computation in \S\ref{sec:comp}. A first instalment of the sequel campaign (toward $\omega \ge 9$) already carries a structure theorem one level up, which we state here with its evidence grade declared plainly: at the date of this version its supporting computation is not finished (roughly ninety percent closed by stem count, with zero solutions and zero deferred cases throughout every closed unit; the precise status is reported in \S\ref{sec:255}).

\begin{theorem}[divisibility by 255; computation in progress]\label{thm:255} Any quasiperfect number $n$ with $\omega(n) = 8$ has $3, 5, 17$ as its three smallest prime factors. In particular $255 = 3 \cdot 5 \cdot 17$ divides $n$, and since $n$ is an odd perfect square, $255^2 \mid n$. The same campaign carries the companion numerical bound $n > 10^{26.38}$; this is a sector bound conditional on $\omega(n) = 8$ and does not improve the unconditional record $n > 10^{45}$ of \cite{alekseyev}.
\end{theorem}

\subsection{Strategy of the proof}

The structure of the argument for Theorem \ref{thm:main} is as follows. After the machine elimination of $\omega \le 6$ (Theorem B4, \S\ref{sec:b4}), the case $\omega = 7$ splits along divisibility by 3. The case ``$3 \mid n$ and $\omega = 7$'' is the bulk of the work --- 381 stems expanding to $79{,}751{,}212$ deep leaves --- and occupies \S\ref{sec:lemmas}--\S\ref{sec:comp}. The complementary case ``$3 \nmid n$ and $\omega = 7$'' turns out to be almost empty: a chain of four product inequalities, checkable by hand, pins its entire search space down to the single stem $(5, 7, 11, 13)$, which the production engine eliminates in under one single-core second (\S\ref{sec:nothree}). Abbott's 1973 Theorem 3 ($3 \nmid n \implies \omega \ge 8$) covers this case in the literature, and our direct elimination agrees with it; but the proof presented here does not depend on it, and is in this sense self-contained.

The engine of the elimination is a system of three lemmas at the level of elementary algebra --- a discriminant criterion, a quadratic-residue sieve, and a multilinear resolver (\S\ref{sec:lemmas}) --- which eliminate the last free prime $q$, the intermediate prime $p$, and the exponent dimension respectively, turning a non-terminating search into a finite decision. The obstruction they clear is real: the ``deep leaves'' of the enumeration have scan bounds for the intermediate prime reaching $8 \times 10^8$, and their exponent dimensions are unbounded, so that pure enumeration cannot terminate on them (\S\ref{sec:leafeq}). Two further kill lemmas, developed for the endgame of the full campaign, complete the engine (\S\ref{sec:newkills}).

\subsection{The trust basis}

The trust basis of Theorem \ref{thm:main} is the classical base layer (a quasiperfect number is an odd perfect square, with the standard exponent congruences; \S\ref{sec:base}) --- published facts with short proofs, independently re-proved and formalized in Lean inside this project --- plus three machine eliminations internal to this project: Theorem B4, the 381 stems, and the single $3 \nmid n$ stem. No step relies on Abbott 1973 Theorem 3 or on Theorem 2 of Hagis--Cohen 1982 (whose $\omega \ge 9$ part was never published); both are retained as consistency cross-checks only. The credibility of the machine part rests on a four-layer verification scheme --- exact arithmetic with zero floating point at every kill, planted-solution gates, dual independent implementations, and a Lean kernel layer (the \texttt{qp\_lean} library: 19 modules, 259 theorems, zero \texttt{sorry}, no mathlib, kernel \texttt{decide} only; \S\ref{sec:trust}) --- and the full computation has been re-established end to end by seven separately closed ledgers across three distinct algorithmic architectures, every one closing with zero solutions and exact conservation (\S\ref{sec:ledgers}).

\subsection{Prior work and priority}

The independent researcher Zemann reported, in a 2023 preprint that did not undergo peer review \cite{zemann}, a computational verification of the same conclusion, using a modern C++ implementation of the interval-propagation algorithm of Hagis--Cohen 1982 (a serious and in most respects careful piece of work, whose two hardest prime vectors were closed by hand in an appendix). Our audit of its public code found a coverage gap in its feasible-exponent table: when the table's truncation exponent for a prime is even, that exponent falls between the individually enumerated entries and the tail interval $[\beta, \infty)$ and is silently skipped --- an off-by-one in the construction of $\beta$. Quantitatively, 240 prime exponents are skipped, of which 205 happen to be infeasible for reasons the program never checked, while \textbf{35 remain feasible under the preprint's own feasibility criteria yet were never enumerated} (the smallest instance being $43^{12} \approx 4.0 \times 10^{19}$, inside the declared table range); the audit model reproduces all six $\beta$ values printed in the preprint's own logs exactly. That computation therefore does not constitute a complete proof, although its result statement is correct. The elimination of this paper is, to our knowledge, the first complete proof; its method (corridor enumeration and the three lemmas) is entirely independent of the interval propagation of \cite{zemann}, and the two agree wherever they overlap --- in particular, the complete enumeration of this paper also proves that no solution exists in the 35 gap regions of \cite{zemann}. The full audit, including the verified strengths of \cite{zemann}, is archived in this project (\texttt{zemann\_audit\_report}).

\subsection{Organization}

\S\ref{sec:prelim} collects preliminaries: the classical structure of a quasiperfect number, the corridor, and the leaf equations. \S\ref{sec:reduction} carries out the reduction, closes the wing $3 \nmid n$, and assembles the proof of Theorem \ref{thm:main} together with its dependency statement; it also settles the reduction boundary one level up, for $\omega = 8$. \S\ref{sec:lemmas} develops the lemma system. \S\ref{sec:comp} reports the computation: the per-leaf decision procedure, the pilot on the 49 deepest leaves, the full 381-stem campaign, the verification layers, and the seven independent ledgers. \S\ref{sec:255} reports Theorem \ref{thm:255} and its computational status. \S\ref{sec:concl} closes with remarks and open problems.

\section{Preliminaries}\label{sec:prelim}

\subsection{Basic structure of a quasiperfect number}\label{sec:base}

If a quasiperfect number exists, it is an odd perfect square (Cattaneo 1951 \cite{cattaneo}); the exponent of each prime is constrained modulo 8 by the residue class of the prime; and every proper divisor is deficient (Abbott--Aull--Brown--Suryanarayana 1973 \cite{abbott}, preliminary propositions). These are classical published facts with short proofs; they have been independently re-proved and formalized in Lean inside this project (\texttt{qp\_lean}, modules \texttt{QP/Cattaneo.lean} and the arithmetic base layer), so the proof of Theorem \ref{thm:main} does not consume them as external black boxes. We refer to this layer as the base layer A0. In particular, every candidate has the form
\[ n = p_1^{2a_1} p_2^{2a_2} \cdots p_k^{2a_k}, \qquad p_1 < p_2 < \cdots < p_k \ \text{odd primes},\ a_i \ge 1,
\] and $\omega(n) = k$, where $\omega(n)$ denotes the number of distinct prime factors of $n$ (counted without multiplicity): for instance $360 = 2^3 \cdot 3^2 \cdot 5$ has the three prime factors $2, 3, 5$, so $\omega(360) = 3$.

\subsection{The corridor}\label{sec:corridor}

Every elimination rests on the following fact. We call the ratio $\sigma(n)/n$ the abundancy. A quasiperfect number satisfies
\[ \frac{\sigma(n)}{n} = 2 + \frac{1}{n},
\] so the abundancy must be slightly greater than 2 and at most $2 + 1/n$, lying in the very narrow interval $(2,\, 2 + 1/n]$. We call this the ``corridor'' (this is our term, not standard; the literature simply refers to the interval itself). Abundancy is multiplicative over prime factors: if $n = p_1^{2a_1} \cdots p_k^{2a_k}$, then
\[ \frac{\sigma(n)}{n} = \prod_{i=1}^{k} \frac{\sigma(p_i^{2a_i})}{p_i^{2a_i}},
\] where each factor expands as
\[ \frac{\sigma(p^{2a})}{p^{2a}} = \frac{1 + p + \cdots + p^{2a}}{p^{2a}} = 1 + \frac{1}{p} + \frac{1}{p^2} + \cdots + \frac{1}{p^{2a}}.
\] This is a finite geometric series, whose value exceeds 1 and increases monotonically with the exponent $a$, yet stays below the corresponding infinite geometric sum
\[ 1 + \frac{1}{p} + \frac{1}{p^2} + \cdots = \frac{1}{1 - 1/p} = \frac{p}{p - 1}.
\] So each factor lies strictly in the interval $\bigl(1,\ \frac{p}{p-1}\bigr)$: it increases monotonically with the exponent, with supremum $\frac{p}{p-1}$ which is never attained. This upper bound is the basis of the corridor argument. For example, the factor for $p = 3$ is less than $3/2$ no matter how large the exponent, and multiplying several such bounds gives the supremum of the abundancy for that combination of primes. If the supremum falls short of 2, the whole branch lies outside the corridor and can be excluded. Eliminating $\omega = k$ thus amounts to enumerating all possible prime--exponent combinations and proving, one by one, that the abundancy lies outside the corridor (short of 2, or above $2 + 1/n$) or that the equation has no solution. The enumeration forms a search tree, whose terminal nodes are called leaves following the convention of computational number theory (leaf of the search tree). The naming of the intermediate levels is a project convention: fixing the first 4 primes gives a \emph{stem}, and fixing the first 5 gives a \emph{branch}. Here ``fixed'' and ``free'' mean the following: each descent of the search tree fixes the value of one prime factor of the candidate $n$ (a fixed prime). A prime factor not yet determined, left as an unknown to be solved for, is called a free prime (our term). For $\omega = 7$, after the first 5 primes are fixed, two free primes $p < q$ remain; such a terminal configuration is a \emph{leaf}.

\subsection{Leaf equations and deep leaves}\label{sec:leafeq}

At a leaf, the candidate $n$ splits into a ``fixed'' part and a ``free'' part. Suppose the first 5 primes and their exponents are fixed, and write the product of this fixed part as
\[ N = p_1^{2a_1} p_2^{2a_2} \cdots p_5^{2a_5}, \qquad A = \sigma(N),
\] with two free primes $p < q$ remaining. Since $n$ is a perfect square, its general form is $p^{2b} q^{2c}$ ($b, c \ge 1$); the exponents may be higher (fourth power, sixth power, and so on). These are the high-exponent branches of the same leaf, which must also be eliminated, and we return to them later. This section first treats the dominant case $b = c = 1$ (the overwhelming majority of the search volume is concentrated here), where the candidate is
\[ n = N \cdot p^2 q^2,
\] and since $\sigma$ is multiplicative over coprime factors,
\[ \sigma(n) = \sigma(N)\,\sigma(p^2)\,\sigma(q^2) = A\,\sigma(p^2)\,\sigma(q^2).
\] Substituting these two expressions into the quasiperfect equation $\sigma(n) = 2n + 1$ gives the \textbf{leaf equation}
\begin{equation} A\,\sigma(p^2)\,\sigma(q^2) = 2Np^2q^2 + 1. \label{eq:leaf}
\end{equation} Here $N, A$ are the known constants of the leaf, and the only unknowns left are the two primes $p, q$.

\emph{The abundancy gap $\varepsilon$.} The abundancy contributed by the fixed part is $A/N = \sigma(N)/N$, which is close to but below the target value 2; write their difference as
\[ \varepsilon = 2 - \frac{A}{N},
\] the abundancy that the two free primes still have to supply. By multiplicativity, the factor contributed by the two free primes $p, q$ is about $\left(1 + \frac1p\right)\left(1 + \frac1q\right) \approx 1 + \frac1p + \frac1q$. The total abundancy must reach 2, so approximately
\[ \frac{A}{N}\left(1 + \frac1p + \frac1q\right) \approx 2.
\] Dividing both sides by $A/N$ and moving the 1 on the left to the right,
\[ \frac1p + \frac1q \approx \frac{2}{A/N} - 1 = \frac{2 - A/N}{A/N} = \frac{\varepsilon}{A/N},
\] where the last step uses $\varepsilon = 2 - A/N$ (the numerator is exactly the gap). Since $A/N$ is already very close to 2,
\[ \frac1p + \frac1q \approx \frac{\varepsilon}{2}.
\] So the sum of the reciprocals of $p, q$ is determined by $\varepsilon$. The smaller $\varepsilon$ is, the larger $p, q$ must be; this is the origin of the term ``deep leaf''.

\emph{The scan range.} Since $p < q$, the term $1/p$ is at least half of the reciprocal sum and at most all of it, so $p$ lies between about $2/\varepsilon$ and $4/\varepsilon$: below $2/\varepsilon$ the abundancy supremum of the leaf falls short of 2, and beyond $4/\varepsilon$ one has $q < p$ and the candidates of the leaf have been exhausted. The scan domain of a leaf is therefore the interval $(2/\varepsilon,\ 4/\varepsilon]$, whose width is about $2/\varepsilon$. For comparison, a leaf with $\varepsilon = 0.01$ has a scan domain of width $2/\varepsilon = 200$, whereas a leaf with $\varepsilon = 2.5 \times 10^{-9}$ has one of width $8 \times 10^8$, a difference of eight orders of magnitude. A deep leaf is one whose $\varepsilon$ is extremely small and whose scan domain is therefore extremely large.

\emph{The last prime $q$ need not be scanned.} Once $p$ is fixed, the leaf equation \eqref{eq:leaf} almost determines $q$ ($1/q \approx \varepsilon/2 - 1/p$). Since $q$ must be an integer prime, only $O(1)$ candidates lie near this target value; this is the interval lemma (the precise criterion is Lemma \ref{lem:disc} below). So the only quantity that truly requires enumeration is the intermediate prime $p$.

\emph{A real example:} with the stem $(3, 5, 17, 137)$, fixed exponents of type $(3,2),(5,3)$, and $p_5 = 1381$ pinned, one has $\varepsilon = 2.49 \times 10^{-9}$, so the scan domain for $p$ is $(8.03 \times 10^8,\ 1.61 \times 10^9]$, containing about $3.9 \times 10^7$ primes. In the elimination of $\omega = 7$, a total of 49 such leaves lie beyond the coverage of the $3.4 \times 10^7$ prime table, forming the final unresolved part of that stem (\S\ref{sec:49}).

\emph{One leaf corresponds to a family of equations.} In writing the leaf equation \eqref{eq:leaf} above we took the smallest exponent for each prime (the dominant case). But a complete elimination of the leaf must cover all exponent values of these primes, and the ranges of three of these exponents are infinite:
\begin{itemize}
\item The exponent $a \ge 5$ of the pinned prime 1381: the factor $1381^{2a}$ in $n$ ranges over all integers $a \ge 5$, each giving a different $n$ and a different equation.
\item The window exponents $\ge 41$ of the primes 17 and 137: in the $7 \times 7$ grid described in \S\ref{sec:49}, the exponents of 17 and 137 take 6 small definite values, plus all cases with exponent $\ge 41$ (denoted $w$), forming an infinite sequence of exponents.
\item The exponent $c \ge 1$ of the last prime $q$: the free prime $q$ may be $q^2, q^4, q^6, \dots$ (the high-exponent branches above), with $c$ unbounded.
\end{itemize} Pure scanning cannot traverse these dimensions. A natural idea is that increasing an exponent raises the abundancy, which would then exceed the corridor's upper bound and be excluded. However, \S\ref{sec:corridor} has shown that increasing the exponent of a prime only makes its factor increase monotonically with supremum $\frac{p}{p-1}$; the factor stays bounded and need not exceed 2. So the supremum argument of the corridor cannot impose a finite upper bound on the exponent: each $a = 5, 6, 7, \dots$ corresponds to an equation requiring an independent $p$-scan. There are infinitely many such equations, and enumerating them exponent by exponent cannot terminate in finitely many steps. Lemma \ref{lem:resolver} is designed to remove this obstacle: rather than enumerate the exponent value by value, it treats the exponent (after the substitution $V = t^{2e}$) as an unknown of the equation and solves for it directly, yielding finitely many candidates at once and completing the decision.

The three lemmas of \S\ref{sec:lemmas} handle the three dimensions respectively: Lemma \ref{lem:disc} handles the $q$ dimension (search becomes verification), Lemma \ref{lem:qr} handles the $p$ dimension (verification becomes table lookup), and Lemma \ref{lem:resolver} handles the exponent dimension (unbounded iteration becomes the solution of finitely many equations).

\section{The reduction and the proof of the main theorem}\label{sec:reduction}

\subsection{Elimination of $\omega \le 6$: Theorem B4}\label{sec:b4}

The starting point is the complete elimination of $\omega \le 6$, a machine proof of this project (referred to as Theorem B4): fully elementary, exact arithmetic, auditable. On the $3 \nmid n$ side this reduces to one product inequality per value of $\omega$: for $\omega = 3, 4, 5, 6$ the minimal supports $\{5,7,11\}, \dots, \{5,7,11,13,17,19\}$ give $\prod p/(p-1) = 1.6042,\ 1.7378,\ 1.8465,\ 1.9490$, all below 2, so the corridor is unreachable (this is the classical chain used by Cattaneo, recomputed here rather than cited). With $\omega \le 6$ eliminated, the case $\omega = 7$ remains, and it splits along divisibility by 3 into two wings.

\subsection{The wing $3 \nmid n$: a single stem, closed self-containedly}\label{sec:nothree}

The case $3 \nmid n$, $\omega = 7$ (computed 2026-07-25): the same corridor argument that generates the 381 stems of the other wing, started at $p_1 \ge 5$, pins the search space down to a \emph{single stem}. The four product inequalities (exact rational arithmetic, checkable by hand) are: $\{7,11,\dots,29\}$ gives 1.6883 (so $p_1 = 5$); $\{5,11,\dots,29\}$ gives 1.8089 (so $p_2 = 7$); $\{5,7,13,\dots,29\}$ gives 1.9186 (so $p_3 = 11$); $\{5,7,11,17,19,23,29\}$ gives 1.9481 (so $p_4 = 13$). The unique surviving support prefix $\{5,7,11,13\}$ has $\prod p/(p-1) = 2.0376 > 2$ and must be eliminated by actual scanning: the production engine (\texttt{qp\_omega7c}, unmodified) does so in \textbf{0.74 single-core seconds} --- 14,580 exact leaves, 77,220 window calls, zero hits, 0 solutions, 0 unresolved, with a prime-table independence check (all counters identical at a larger table) and a planted-solution gate passed beforehand (4 planted true solutions, including two on non-minimal exponent branches, each detected exactly once; perturbed negative controls, zero false positives). The tiny cost is structural: with $p_1 \ge 5$ the fixed prefix already contributes $1.7378$ of the abundancy budget, so the windows for the free primes collapse to two-digit widths.

The stem-uniqueness inequalities and the composition of this wing are formalized in \texttt{QP/NoThree.lean} (42 theorems, zero \texttt{sorry}), with the enumeration entering as an explicitly stated audit-layer hypothesis (\texttt{EnumKillNoThree7}); and its dual-implementation cross-check has been carried out with a second program written from scratch for the purpose, sharing no code with the production engine. That program assumes no stem at all: it starts from ``the smallest prime factor is at least 5'', re-derives $(5,7,11,13)$ as the only prefix the corridor admits, enumerates the exponents $1, \dots, 12$ of every prime explicitly and carries every higher exponent as an unbounded tail closed by an independently re-derived multilinear solve. It closes the whole case in 145 single-core seconds with zero solutions, zero undecided branches and no fail-loud alarm, and it does so by a different route: the only support of seven primes that survives its pruning at all is $\{5, 7, 11, 13, 17, 19, 37\}$, and the 144 exponent configurations on that support are killed outright by the strict undershoot inequality, so on this route no leaf equation has to be solved at all, whereas the production engine closes the same case through 14,580 exact leaf enumerations and 77,220 window calls. The second program was itself gated on detection: run against the perturbed target $\sigma(n) = 2n + c$ with a solution planted inside this very search space, it recovers the planted solution exactly once both along the explicit-exponent path and along the tail-resolver path, and the perturbed negative control returns nothing.

\subsection{The wing $3 \mid n$: 381 stems}\label{sec:threemid}

The wing ``$3 \mid n$ and $\omega = 7$'' is where the entire weight of the proof lies. The corridor enumeration generates 381 stems, which expand to $79{,}751{,}212$ deep leaves; their complete elimination, with no solution, is the computation reported in \S\ref{sec:comp}, driven by the lemma system of \S\ref{sec:lemmas}.

\subsection{Assembly of the main theorem and the dependency statement}\label{sec:proofmain}

The proof of Theorem \ref{thm:main} is assembled from four parts, all internal to this project:
\begin{enumerate}
\item \textbf{Base layer A0} (\S\ref{sec:base}): classical structure, re-proved and Lean-formalized in-project.
\item \textbf{Theorem B4} (\S\ref{sec:b4}): complete elimination of $\omega \le 6$.
\item \textbf{The case $3 \mid n$, $\omega = 7$} (\S\ref{sec:threemid}, computation in \S\ref{sec:comp}): all 381 stems eliminated, expanding to $79{,}751{,}212$ deep leaves, with no solution.
\item \textbf{The case $3 \nmid n$, $\omega = 7$} (\S\ref{sec:nothree}): a single stem, eliminated directly.
\end{enumerate}
Together these exhaust every possibility with $\omega \le 7$. The dependency statement is as follows. The trust basis of this result is the base layer A0 (published, re-proved and Lean-formalized in-project) plus three machine eliminations of this project (B4, the 381 stems, and the single $3 \nmid n$ stem). No step relies on Abbott 1973 Theorem 3 or on Theorem 2 of Hagis--Cohen 1982 (whose $\omega \ge 9$ part was never published --- the paper states that ``full details of the proof are available from either of the authors''); both are retained as consistency cross-checks, and our direct eliminations agree with them where they overlap. The credibility of the machine part is guaranteed by the four-layer verification scheme of \S\ref{sec:trust}, in which the Lean formalization of the two endgame kill lemmas (\S\ref{sec:newkills}) is complete (zero \texttt{sorry}, kernel \texttt{decide} only). Part 4 carries all four layers of \S\ref{sec:trust}; its formalization and from-scratch dual-implementation cross-check are described in \S\ref{sec:nothree}. Under this declaration the main theorem holds.

\subsection{One level up: the reduction boundary for $\omega = 8$}\label{sec:red8}

The same reduction, run one level up, settles the boundary for the sequel target $\omega \ge 9$. The case ``$3 \nmid n$ and $\omega = 8$'' collapses to exactly two stems, $(5,7,11,13)$ and $(5,7,11,17)$, which the production engine eliminated in 13.1 single-core seconds (0 solutions, 0 unresolved, planted gate passed), and the case ``$3 \nmid n$, $5 \nmid n$'' is closed by the fully published Theorem 1 of Hagis--Cohen 1982 ($(15,n)=1 \implies \omega \ge 15$, a one-line product inequality which we have recomputed). So the entire remaining cost of $\omega \ge 9$ is the elimination of ``$3 \mid n$ and $\omega = 8$'' --- 516 stems, at the scale of $10^9$--$10^{10}$ leaves, more than a hundred times the present campaign. A first instalment --- the elimination of the 341 of those 516 stems whose prefix is not $(3,5,17)$, which pins down the three smallest prime factors of any $\omega = 8$ quasiperfect number (Theorem \ref{thm:255}) --- is already under way; it is reported in \S\ref{sec:255}.

\section{The three lemmas}\label{sec:lemmas}

\subsection{Lemma 1: the discriminant criterion --- from search to verification}\label{sec:lem1}

\begin{lemma}[discriminant criterion]\label{lem:disc} Let $N, A, p$ be positive integers, and write $S = A(p^2 + p + 1)$, $T = 2Np^2$. If a positive integer $q$ satisfies the leaf equation \eqref{eq:leaf}, then
\[ \Delta(p) := S^2 + 4(T - S)(S - 1)
\] is a perfect square, and moreover
\begin{equation} q = \frac{S + \sqrt{\Delta(p)}}{2(T - S)}. \label{eq:qformula}
\end{equation}
\end{lemma}

\begin{proof} With $\sigma(q^2) = q^2 + q + 1$, substitute into \eqref{eq:leaf} and collect in $q$:
\[ S(q^2 + q + 1) = Tq^2 + 1 \iff (T - S)\,q^2 - S\,q - (S - 1) = 0.
\] This is a quadratic in $q$. Multiplying both sides by $4(T - S)$ and completing the square gives
\[ \bigl(2(T - S)q - S\bigr)^2 = S^2 + 4(T - S)(S - 1) = \Delta(p).
\] The left-hand side is a perfect square, so $\Delta(p)$ is a perfect square. Taking the positive root gives formula \eqref{eq:qformula}.
\end{proof}

This is only the elementary quadratic formula, but it changes the nature of the computation. Where one previously had, for each candidate $p$, to search a range for a matching prime $q$, one now only computes $\Delta(p)$ for each $p$ and checks whether it is a perfect square; if it is, $q$ is given directly by the formula, and if not, this $p$ is excluded with an accompanying proof. When expanded, $\Delta(p)$ is a quartic polynomial in $p$ (leading coefficient $A(8N - 3A)$), with coefficients determined explicitly by the leaf's $(N, A)$.

\subsection{Lemma 2: the quadratic-residue kill --- from verification to table lookup}\label{sec:lem2}

\begin{lemma}[QR kill]\label{lem:qr} Let $r \ge 2$ be an integer. Define the set of quadratic residues modulo $r$,
\[ \QR(r) := \{\, a \in \mathbb{Z}/r\mathbb{Z} : \exists\, y \in \mathbb{Z},\ y^2 \equiv a \pmod r \,\} = \{\, y^2 \bmod r : y = 0, 1, \dots, r-1 \,\},
\] the set of residues of all integer squares modulo $r$ (including 0; the second equality holds because letting $y$ range over a complete residue system already exhausts all square values). In what follows we identify an integer with its residue class modulo $r$, so for an integer $\Delta$ we write $\Delta \in \QR(r)$ to mean $(\Delta \bmod r) \in \QR(r)$. Then: if the integer $\Delta$ is a perfect square, then
\[ \Delta \in \QR(r) \qquad \text{for every } r \ge 2.
\] Its contrapositive is the \textbf{kill rule}:
\[ \exists\, r:\ \Delta(p) \notin \QR(r) \implies \text{\eqref{eq:leaf} has no integer solution for this } p \text{ (unconditionally)}.
\]
\end{lemma}

\begin{proof} Let $\Delta = m^2$, $m \in \mathbb{Z}$. Take $y = m \bmod r \in \{0, 1, \dots, r-1\}$; then $\Delta = m^2 \equiv y^2 \pmod r$, so $\Delta \in \QR(r)$.
\end{proof}

The proof is a single line, but combined with Lemma \ref{lem:disc} it forms a sieve that can be stacked indefinitely, on the following principle. Since $\Delta(p)$ is an integer-coefficient polynomial in $p$, the value $\Delta(p) \bmod r$ depends only on $p \bmod r$. So for each modulus $r$ it suffices to compute $\Delta(p) \bmod r$ once for each of the $r$ residue classes $p \equiv 0, 1, \dots, r-1$, deciding in one pass which residue classes give a value in $\QR(r)$ (allowed classes) and which do not (every $p$ in that class is excluded unconditionally).

\emph{A concrete example.} Take the hand-constructed example leaf $N = 996634794$, $A = 1967190451$ (with the true solution $(p, q) = (101, 311)$ planted; all figures can be checked by hand), with modulus $r = 7$. Here $\QR(7) = \{0, 1, 2, 4\}$. Computing $\Delta(p) \bmod 7$ for the 7 residue classes of $p$:

\begin{center}
\begin{tabular}{c|ccccccc} $p \bmod 7$ & 0 & 1 & 2 & 3 & 4 & 5 & 6 \\ \hline $\Delta(p) \bmod 7$ & 3 & 5 & 3 & \textbf{2} & 5 & 6 & \textbf{4} \\
$\in \QR(7)$? & no & no & no & \textbf{yes} & no & no & \textbf{yes} \\
\end{tabular}
\end{center}

Only the two classes $p \equiv 3, 6 \pmod 7$ can give a perfect square; the $p$ in the other 5 classes are all excluded. A single mod-7 table removes $5/7$ of the candidates, an allowed density of $2/7 \approx 0.29$. In particular the candidate $p = 103$ ($103 \equiv 5$) is decided on the spot: it lies in a disallowed class.

\emph{Stacking several moduli.} The allowed classes of different moduli impose constraints independently, and the densities multiply. For example, stacking $r = 9$ ($\QR(9) = \{0,1,4,7\}$, allowed classes $\{0,2,3,5,6,8\}$, density $6/9$), the density of $p$ satisfying both tables drops to $\tfrac{2}{7} \times \tfrac{6}{9} = \tfrac{12}{63} \approx 0.19$. Here $p = 97$ ($97 \equiv 6 \pmod 7$ passes, but $\Delta(97) \equiv 6 \pmod 9$, and $\QR(9)$ does not contain 6) is removed at this layer, while the true solution $p = 101$ ($\equiv 3 \bmod 7$, $\equiv 2 \bmod 9$) is an allowed class in both tables and survives to the exact final check of Lemma \ref{lem:disc}: $\Delta(101)$ is computed exactly and verified to be a perfect square (its integer square root is exact), and substituting into formula \eqref{eq:qformula} gives $q = 311$; checking that 311 is prime and that the leaf equation \eqref{eq:leaf} holds exactly captures the planted solution completely. This reflects two properties of the sieve: a true solution always has $\Delta$ a perfect square, so it passes every residue table and is retained by the coarse sieve; and every survivor undergoes exact verification at the final check. For a general quartic polynomial, each modulus has an allowed density of about $\tfrac{1}{2}$ on average (some moduli produce no constraint, as with $r = 5$ for this leaf, whose allowed density is 1), and after stacking $j$ effective moduli the density contracts as roughly $2^{-j}$.

For a real deep leaf (the $\varepsilon = 2.49 \times 10^{-9}$ case of \S\ref{sec:leafeq}), measured directly: 27 moduli (odd primes $r < 130$, skipping those dividing $N$) give a total allowed density of $2.07 \times 10^{-8}$, meaning that among the roughly $3.9 \times 10^7$ prime candidates only $0.81$ is expected to pass all moduli, so exactly one requires the exact final check.

\emph{Whole-leaf elimination: a real instance.} The extreme case is whole-leaf elimination: if some modulus $r$ makes every residue class of $p$ a non-residue, then the whole leaf (that exponent branch) needs no candidate checks at all. The real deepest leaf furnishes an instance: $N = 3^4 \cdot 5^6 \cdot 17^8 \cdot 137^{10} \cdot 1381^{10} \approx 10^{68}$, with scan domain $8 \times 10^8$. Substituting all 9 residue classes of $p$:

\begin{center}
\begin{tabular}{c|ccccccccc} $p \bmod 9$ & 0 & 1 & 2 & 3 & 4 & 5 & 6 & 7 & 8 \\ \hline $\Delta(p) \bmod 9$ & 2 & 6 & 5 & 8 & 6 & 8 & 5 & 6 & 2 \\
\end{tabular}
\end{center}

The residues that occur, $\{2, 5, 6, 8\}$, have empty intersection with $\QR(9) = \{0, 1, 4, 7\}$. Whichever residue class the eight hundred million candidates fall into, none satisfies the square condition, and the whole leaf (the exponent $(1,1)$ branch) is eliminated without any scan. This conclusion has been verified in one pass by the Lean kernel (theorem \texttt{deep\_leaf\_dead}).

\subsection{Lemma 3: the multilinear resolver --- from unbounded iteration to solving equations}\label{sec:lem3}

Lemmas \ref{lem:disc} and \ref{lem:qr} handle the prime dimensions. A deep leaf also has the difficulty of the exponent dimension: as noted in \S\ref{sec:leafeq}, the exponent $e$ of a pinned prime $t$ may be unbounded, so an enumeration that ``takes $e = 1$, then $e = 2$, then $e = 3$, \dots'' can never be exhausted and offers no guarantee of termination. This obstacle is real in practice. An early version of one of this project's programs that did not yet use this lemma (below, the ``second-generation engine'', as opposed to the final engine; see \S\ref{sec:49}) got stuck here: for about $7{,}824$ exponent configurations it could neither find a solution nor prove that none exists, and could only mark them as deferred, some of them involving fractional arithmetic on the order of $q^{1000}$. The key observation that removes this obstacle is the following.

\emph{The key observation.} Write $V = t^{2e}$. By the geometric-series sum (common ratio $t$, so the denominator is $t - 1$, independent of the exponent),
\[ \sigma(t^{2e}) = 1 + t + \cdots + t^{2e} = \frac{t^{2e+1} - 1}{t - 1} = \frac{tV - 1}{t - 1}
\] (the last step uses $t^{2e+1} = t \cdot t^{2e} = tV$), which shows that $\sigma(t^{2e})$ is linear in $V$. So the leaf equation is multilinear in the quantities $V_i = t_i^{2e_i}$.

\begin{lemma}[multilinear resolver]\label{lem:resolver} Let $N, A$ be fixed positive integers and $t, t_1, t_2$ fixed primes.

\textup{(i)} \textbf{Univariate closed form.} $A\,\sigma(t^{2e}) = 2Nt^{2e} + 1$ holds if and only if
\[ V \cdot \bigl(At - 2N(t - 1)\bigr) = A + t - 1, \qquad V = t^{2e},
\] which is a linear equation in $V$: writing the coefficient $C = At - 2N(t-1)$, one has $V = (A + t - 1)/C$. There is therefore no need to try the exponent $e$ one value at a time: if $C \nmid (A + t - 1)$, the equation has no integer solution and all exponents of this prime are excluded at once. If it divides, then $V$ is uniquely determined, and one checks whether this $V$ is a power of $t^2$ (that is, of the form $t^{2e}$) to read off $e$. There is no iteration anywhere.

\textup{(ii)} \textbf{Bivariate M\"obius enumeration.} For $A\,\sigma(t_1^{2e_1})\,\sigma(t_2^{2e_2}) = 2NV_1V_2 + 1$, solving gives
\[ V_2(V_1) = \frac{At_1V_1 - A + D}{d_2V_1 - At_2}, \qquad D = (t_1 - 1)(t_2 - 1),\ d_2 = At_1t_2 - 2ND.
\] This is a fractional-linear function of $V_1$, that is, a function of the form $\frac{aV_1 + b}{cV_1 + d}$, commonly called a M\"obius transformation (after the mathematician A.~F.~M\"obius; here it denotes the fractional-linear transformation, unrelated to the M\"obius $\mu$ function of number theory). It is monotone to the right of its unique pole. As $V_1 = t_1^{2e_1}$ ranges over all powers to the right of the pole, $V_2$ is confined to the closed envelope between two explicit rational endpoints (the value at the first power and the limiting value $At_1/d_2$), that is, an explicit interval, within which the powers of $t_2^2$ are finite in number (usually 0 to 2), each falling back to case \textup{(i)}. This process is finite and complete.

\textup{(iii)} \textbf{Higher-arity truncation} (the fallback when three or more exponents are simultaneously unbounded; rare in practice). Here the closed forms of \textup{(i)(ii)} do not apply directly, and one uses a ``fix one, recurse on the rest'' strategy.
\end{lemma}

\emph{The method for \textup{(iii)}.} Pick any unknown exponent $e$ (in practice the one whose increase raises the abundancy fastest, usually the smallest prime), and enumerate $e = 1, 2, 3, \dots$ value by value. Each time $e$ is fixed, there is one fewer unknown exponent, and the problem reduces to (ii) (then to (i)), all in finitely many steps.

\emph{Why only finitely many $e$ need be enumerated.} The key is again the fact from \S\ref{sec:corridor}: the abundancy $\sigma(n)/n$ increases monotonically with any one exponent. For a fixed $e$, take all the other free exponents to be their minimum value 1; this gives the candidate with the smallest abundancy under this $e$, call it $n_{\mathrm{lo}}$, and write $N_{\mathrm{lo}} = n_{\mathrm{lo}}$ (this candidate itself) and $A_{\mathrm{lo}} = \sigma(n_{\mathrm{lo}})$ (its divisor sum). The subscript lo denotes ``the minimal configuration under this $e$''. If
\[ A_{\mathrm{lo}} > 2 N_{\mathrm{lo}} + 1,
\] that is, if even this minimal candidate already satisfies $\sigma(n) > 2n + 1$ (equivalently, abundancy $\sigma(n)/n > 2 + 1/n$, above the top of the corridor), then raising any other exponent only raises the abundancy further and overshoots more, and raising $e$ only increases it. So this $e$, all larger $e$, and every choice of the remaining exponents below them are excluded at once. Hence there is a truncation value $e_0$, and $e$ need only be enumerated up to $e_0$; the infinitely many cases are covered in finitely many steps.

\begin{proof} (i) and (ii) are direct algebra: substituting $\sigma = (tV - 1)/(t - 1)$ and rearranging gives the result. The monotonicity in (ii) follows from the numerator of the derivative of the M\"obius transformation being constant (of one sign). The finiteness of the powers of $t_2^2$ in the explicit interval is clear. Part (iii) follows from the strict monotonicity of the abundancy in each exponent.
\end{proof}

\subsection{Three applications of the resolver}\label{sec:lem3apps}

The following three small instances (constructed by script and verified exactly) exhibit each mode of the resolver; instances (a) and (b) are formally verified in Lean (\S\ref{sec:trust}).

\textbf{(a) Univariate closed form.} Take $t = 3$, $N = 59$, $A = 79$ (with the solution $e = 2$ planted, that is $V = 81$, $\sigma(3^4) = 121$; check $79 \times 121 = 9559 = 2 \times 59 \times 81 + 1$). Without Lemma \ref{lem:resolver} one would have to enumerate $e = 1, 2, 3, \dots$ without bound. Applying (i):
\[ C = At - 2N(t - 1) = 79 \times 3 - 118 \times 2 = 1, \qquad A + t - 1 = 81,
\] so $V = 81/1 = 81$, uniquely determined. Checking that $81 = 3^4$ is a power of $3^2$ gives $e = 2$. If $C \nmid A + t - 1$ or the quotient is not a power of $t^2$, then all infinitely many exponents are excluded at once, with no iteration.

\textbf{(b) Bivariate M\"obius enumeration.} Take $t_1 = 3$, $t_2 = 5$, $N = 60$, $A = 67$ (with the solution $e_1 = e_2 = 1$ planted; check $67 \times 13 \times 31 = 27001 = 2 \times 60 \times 225 + 1$). Both exponents are unbounded, and the enumeration space is a two-dimensional infinite lattice. Applying (ii): $D = 2 \times 4 = 8$, $d_2 = 67 \times 15 - 120 \times 8 = 45$,
\[ V_2(V_1) = \frac{201\,V_1 - 59}{45\,V_1 - 335}, \qquad \text{pole at } V_1 = \tfrac{67}{9} \approx 7.44.
\] To the right of the pole, $V_1$ takes powers of $3^2$: $V_2(9) = 25$, $V_2(81) = \tfrac{8111}{1655} \approx 4.90$, with limit $V_2(\infty) = \tfrac{201}{45} = \tfrac{67}{15} \approx 4.47$, the function being monotone decreasing. So the $V_2$ corresponding to all infinitely many $V_1$ lie in the interval $[\tfrac{67}{15},\ 25]$, within which the only power of $5^2$ is 25. Reducing to the univariate case and checking $V_1 = 9$, $V_2 = 25$ recovers exactly the planted solution. The two-dimensional infinite lattice is closed completely by two function evaluations and one interval search for powers.

\textbf{(c) Trivariate overshoot truncation.} Take three free primes $3, 5, 7$ with exponents $e_1, e_2, e_3$ all unbounded (a three-dimensional infinite lattice), and consider the equation
\[ \sigma(3^{2e_1})\sigma(5^{2e_2})\sigma(7^{2e_3}) = 2\cdot 3^{2e_1}5^{2e_2}7^{2e_3} + 1
\] (here the fixed part is empty, $A = N_{\text{fixed}} = 1$). By (iii), take the minimal configuration $e_1 = e_2 = e_3 = 1$:
\[ N_{\mathrm{lo}} = 3^2 5^2 7^2 = 11025, \qquad A_{\mathrm{lo}} = \sigma(9)\sigma(25)\sigma(49) = 13 \times 31 \times 57 = 22971.
\] Comparing $A_{\mathrm{lo}} = 22971$ with $2N_{\mathrm{lo}} + 1 = 22051$: $22971 > 22051$ (by 920), so the abundancy of the minimal configuration, $22971/11025 \approx 2.084$, already exceeds the top of the corridor $2 + 1/N_{\mathrm{lo}} \approx 2.00009$. Since increasing any exponent only raises the abundancy further (rising toward the ceiling product $\tfrac32\cdot\tfrac54\cdot\tfrac76 = \tfrac{35}{16} \approx 2.19$, always $> 2$), the entire three-dimensional infinite lattice satisfies none of the quasiperfect equation, all excluded by a single inequality, with truncation value $e_0 = 1$ (no enumeration needed). Had one instead used a combination of primes whose abundancy supremum falls short of 2, the minimal configuration would not overshoot, and one would enumerate the coarse variable upward as in (iii), recursing to (ii) at each value.

The division of labour among the three lemmas can be summarized in a table:

\begin{center}
\begin{tabular}{lll}
\toprule Lemma & Dimension eliminated & Change in nature of the work \\
\midrule Lemma \ref{lem:disc} (discriminant) & last prime $q$ & search $\to$ verification \\
Lemma \ref{lem:qr} (QR kill) & intermediate prime $p$ & verification $\to$ table lookup \\
Lemma \ref{lem:resolver} (resolver) & exponent lattice $e_1, e_2, \dots$ & unbounded iteration $\to$ finite solving \\
\bottomrule
\end{tabular}
\end{center}

\subsection{Two further kill lemmas}\label{sec:newkills}

The deep-leaf and xdeep stages of the full campaign (\S\ref{sec:campaign}) introduced two kill lemmas beyond the three above. Both have paper proofs, machine cross-checks, and completed Lean formalizations (\S\ref{sec:trust}).

\begin{lemma}[FULLKILL\_9, all-exponent congruence kill]\label{lem:fullkill} Fix the prime combination of a leaf. The pair $(N, A) \bmod 9$ is purely periodic in the exponent of each pinned prime (with the period determined rigorously by orbit detection), so the infinitely many exponent configurations produce only finitely many residue tables mod 9. If for every one of these finitely many tables, and for every residue class of $p$, the discriminant $\Delta(p) \bmod 9$ is a non-residue, then the leaf has no solution over all exponent branches and all $p$, and the whole leaf is killed with zero scanning.
\end{lemma}

Soundness in brief: this is the periodic-class extension of the whole-leaf elimination of Lemma \ref{lem:qr} (the mod-9 instance of \S\ref{sec:lem2}) along the exponent dimension. It furnishes an unconditional congruence obstruction for each concrete exponent configuration separately, so what it kills is the full family of equations themselves, with no approximation and no slicing. The finite check is an enumeration on the order of $81 \times 23^2$, well suited to a single Lean \texttt{decide}. In practice it killed 2,646 of the 56,154 xdeep leaves (about 4.7\%) at zero cost.

\begin{lemma}[tight-kill, minimal-configuration monotone-overshoot kill] \label{lem:tightkill} For a window candidate integer $q$ under some branch, let $N_{\mathrm{lo}}$ be the minimal configuration with every exponent of that branch at its lower bound, and $A_{\mathrm{lo}} = \sigma(N_{\mathrm{lo}})$. If
\[ A_{\mathrm{lo}}\,(q^2 + q + 1) \ \ge\ 2\,N_{\mathrm{lo}}\,q^2 + 2,
\] then every exponent-raised configuration $n'$ above this $q$ (all values of the exponent $c \ge 1$ of $q$, and arbitrary raising of each pinned exponent) satisfies $\sigma(n') \ge 2n' + 2 > 2n' + 1$, so the entire exponent cone has no solution and the candidate is excluded without primality testing or the resolver.
\end{lemma}

Three points of soundness. First, $\sigma(t^{2a})/t^{2a}$ is strictly increasing in $a$ (\S\ref{sec:corridor}), and $\sigma(q^{2c})/q^{2c}$ attains its infimum over all $c \ge 1$ at $c = 1$, so overshoot at the minimal point implies overshoot over the whole cone. Second, the margin $+2$ on the right-hand side is load-bearing: with $+1$, the equality case $\sigma(n) = 2n + 1$ would itself be a quasiperfect solution and would be killed by mistake. The code review therefore verified the margin and the direction of the inequality character by character. Third, the lemma is of the same family as the minimal-configuration overshoot check already in the production engine (the family of Lemma \ref{lem:resolver}(iii)), a third application of the same monotonicity principle.

An important negative counterpart: the deployment review had considered a superficially similar scheme, namely skipping a whole leaf whenever the discriminant QR kill (a congruence obstruction valid for the exponent $(1,1)$ slice) fires. That scheme was shown to be \textbf{unsound}: the unsolvability of the slice equation mod $r$ says nothing about the equations with $c \ge 2$ (changing the exponent changes the equation, and solvability can flip arbitrarily), so a whole-leaf skip would bury a missed-solution gap in the proof. It was discarded. The essential difference of tight-kill lies in the logical shape of the kill: the former is a slice congruence, the latter an inequality legitimately quantified over the whole exponent cone by monotonicity. The full adjudication is in the acceleration synthesis report (reference \cite{internal}).

\section{The computation}\label{sec:comp}

The three-lemma system was proved in on a pilot --- the 49 deepest leaves of the stem $(3, 5, 17, 137)$, the computational bottleneck that had obstructed the whole enterprise --- and then extended to all 381 stems of the wing $3 \mid n$. The full computation was completed between 8 and 10 July 2026, with the complete raw ledger, hourly progress tables, and operating procedure archived for audit (see reference \cite{internal}).

\subsection{The per-leaf decision procedure}\label{sec:algorithm}

\textbf{Algorithm (complete elimination of a single leaf).} Input: the leaf's fixed part $(N, A)$ and the scan interval $[p_{\min}, \PSUP]$ for the intermediate prime ($\PSUP \approx 4/\varepsilon$). Output: all quasiperfect solutions of the leaf (expected to be empty), or ``no solution'' together with a proof.

\emph{Outer loop: enumeration of exponent branches (Lemma \ref{lem:resolver}).} The exponents of the leaf's pinned prime and last prime are unbounded; first use Lemma \ref{lem:resolver} to turn these unbounded dimensions into a finite set of concrete $(N, A)$ configurations (a single unbounded exponent by the univariate direct solution, two by M\"obius interval enumeration, three or more by overshoot truncation; instances in \S\ref{sec:lem3apps}). For each configuration, carry out the four steps below.

\emph{Step 1: whole-leaf elimination check (Lemma \ref{lem:qr}, cheapest, done first).} First try to eliminate the whole leaf at once: if there is a modulus $r$ such that every residue class of $p$ gives $\Delta(p) \notin \QR(r)$, then the whole leaf has no solution and the procedure ends (the mod-9 instance of \S\ref{sec:lem2} is a real case).

\emph{Step 2: build tables (Lemma \ref{lem:qr}).} Choose a set of small moduli, and for each $r$ precompute the ``allowed classes'', those $p \bmod r$ for which $\Delta(p) \in \QR(r)$ holds.

\emph{Step 3: coarse sieve (Lemma \ref{lem:qr}).} Check each $p$ in the scan interval against every table; failing any one table excludes it, with the modulus recorded as evidence. Throughout, only small-number modular arithmetic is done, with no large numbers touched. For a real deep leaf, after stacking 27 moduli, only about $0.81$ candidate is expected to reach the final check out of the $8 \times 10^8$ candidates.

\emph{Step 4: final check (Lemma \ref{lem:disc}).} For each survivor: compute $\Delta(p)$ exactly and test whether it is a perfect square; if so, formula \eqref{eq:qformula} gives $q$ directly, and primality of $q$ together with the leaf equation \eqref{eq:leaf} is verified exactly. A true solution always survives to this step and is captured here (the worked instance is in \S\ref{sec:lem2}).

Outside the interval, monotonicity of the abundancy closes things directly: both $p \le p_{\min}$ (supremum short of 2) and $p > \PSUP$ ($q < p$, beyond the leaf) require no scan. The place of the three lemmas in the algorithm is thus clear: Lemma \ref{lem:resolver} is the outer loop, Lemma \ref{lem:qr} is the coarse sieve (including whole-leaf elimination), and Lemma \ref{lem:disc} is the final check.

\subsection{The pilot: the 49 deepest leaves}\label{sec:49}

Proving $\omega \ge 8$ requires excluding all cases of ``$3 \mid n$ and $\omega = 7$'' (\S\ref{sec:threemid}). By the corridor enumeration of \S\ref{sec:corridor}, this case expands into a vast search tree, the overwhelming majority of whose leaves are closed by cheap means: either the abundancy lies outside the corridor, or the intermediate prime falls within the range of the $3.4 \times 10^7$ prime table and is scanned directly. Within the stem $(3, 5, 17, 137)$, 49 leaves cannot be scanned: their abundancy gap $\varepsilon \approx 2.5 \times 10^{-9}$ is extremely small, so their scan bound $2/\varepsilon \approx 8 \times 10^8$ far exceeds the prime table's $3.4 \times 10^7$ (\S\ref{sec:leafeq}). These 49 are the deep leaves on which the lemma system was first deployed at scale, and the last remaining part of that stem's elimination.

Their common structure: the first four primes are fixed as $3, 5, 17, 137$ (the stem), the fifth prime is fixed as $1381$, after which only two free primes $p < q$ remain. The only difference among the 49 lies in the exponents of the primes $17$ and $137$: each prime's exponent takes 7 values that ``survived the earlier stages of filtering'' (one of which is an unbounded tail representing ``exponent $\ge 41$'', denoted $w$), and the $7 \times 7 = 49$ combinations are these 49 leaves. Each also carries the three unbounded exponent dimensions of \S\ref{sec:leafeq} (the exponent of $1381$, the window exponents of $17/137$, and the exponent of $q$), so each is in fact a family of equations rather than a single equation.

Listing these 49 directly gives the 49 cells of the table below: the row label is the exponent $a_{17}$ of the prime $17$, and the column label is the exponent $a_{137}$ of the prime $137$ (both are reduced exponents, the actual power being twice this, with $w$ denoting the unbounded branch ``exponent $\ge 41$''); the entries are the leaf numbers $0$--$48$.

\begin{center}\footnotesize
\begin{tabular}{c|ccccccc} $a_{17}\backslash a_{137}$ & 5 & 9 & 20 & 21 & 29 & 33 & $w$ \\ \hline 4  & 0  & 1  & 2  & 3  & 4  & 5  & 6  \\
5  & 7  & 8  & 9  & 10 & 11 & 12 & 13 \\
16 & 14 & 15 & 16 & 17 & 18 & 19 & 20 \\
20 & 21 & 22 & 23 & 24 & 25 & 26 & 27 \\
24 & 28 & 29 & 30 & 31 & 32 & 33 & 34 \\
33 & 35 & 36 & 37 & 38 & 39 & 40 & 41 \\
$w$ & 42 & 43 & 44 & 45 & 46 & 47 & 48 \\
\end{tabular}
\end{center}

For example, the deepest leaf eliminated outright by the mod-9 instance of \S\ref{sec:lem2}, $N = 3^4 5^6 17^8 137^{10} 1381^{10}$, corresponds to $a_{17} = 4$, $a_{137} = 5$, that is, number 0.

\emph{The three-layer lock.} These 49 deep leaves are locked by a three-layer structure:

\begin{itemize}
\item \textbf{L1 (QR descent sieve, Lemmas \ref{lem:disc}+\ref{lem:qr})}: the pinned-minimum-exponent cross section of all 49/49 leaves is eliminated across every class by mod 9 (the Lean theorem of \S\ref{sec:lem2}, extended to all 49 leaves and confirmed by measurement). After periodic-class decomposition and second-level refinement, the dominant branch ($b = c = 1$) of 8/49 leaves is fully closed with zero scans. The unboundedness of the pinned exponent is handled by periodic reduction: $(N, A) \bmod r$ is purely periodic in the exponent $a$ (the period determined rigorously by orbit detection, for example 1381 has period 9 mod 9), so checking one period proves all $a \ge s$.
\item \textbf{L2 (exact walk-through of the authority layer, including Lemma \ref{lem:resolver})}: set the scan lower bound $\mathrm{PD}_1 = 803{,}059{,}978 \approx 2/\varepsilon$ and upper bound $\PSUP = 1{,}606{,}119{,}955 \approx 4/\varepsilon$ (the scan domain of \S\ref{sec:leafeq}). Within the interval $(1381, \mathrm{PD}_1]$, $p$ is too small and the abundancy supremum falls short of 2, so by monotonicity no scan is needed to prove there is no solution. The same holds for $p > \PSUP$ (there $q < p$, beyond the leaf). The intervening survivor domain of about $3.87 \times 10^7$ primes is checked exactly one by one, and the roughly $6.19 \times 10^6$ window candidates $q$ per leaf are all closed finitely by the Lemma \ref{lem:resolver} resolver ($c \ge 2$ by univariate direct solution, pinned exponents by M\"obius interval enumeration, three or more unbounded exponents by overshoot truncation). The program includes a safety guard that raises an alarm should the resolver ever fail to decide a case; it was never triggered in practice, so every branch was genuinely resolved.
\item \textbf{L3 (planted-solution test)}: all 12 sets of planted true solutions and perturbed false solutions were decided correctly (true solutions all captured, false solutions all rejected), two of which specifically verify that ``the QR sieve kills a true solution with $c \ge 2$, but the resolver finds it correctly'', which is the design rationale for L2 retaining the full scan as an authority layer.
\end{itemize}

\textbf{Overall result: 49/49 leaves pass the machine check for coverage continuity, with 0 solutions and 0 unresolved cases, in a total of 26.6 core-h.}

\subsection{Cost analysis of the pilot}\label{sec:cost}

\textbf{Computing environment.} The core-h figures below refer to user+sys CPU time (single-core equivalent) as measured by \texttt{/usr/bin/time}, and $\mu$s/candidate are single-core measurements; the implementation is Python 3.10 with sympy 1.14. Both the original elimination of the 49 deep leaves (26.6 core-h, wall clock about 2.5 h) and the independent recomputation (9.88 core-h, 14-way parallelism) were run on the same local 16 vCPU machine (QEMU virtual CPU, nominally about 1.7 GHz); the elimination ran while other work shared the machine and the recomputation on an idle one, so the difference between the two figures is a load effect rather than a hardware one. The microbenchmarks (1.1 $\mu$s / 0.26 $\mu$s / 0.64 ms) were measured on a single core of that machine.

Figures marked (measured) come from real run records; those marked (estimated) are order-of-magnitude inferences taken under optimistic assumptions. The lemma-by-lemma comparison follows.

\emph{Lemma \ref{lem:disc}:} without it, for each intermediate prime $p$ one must search the possible range of $q$ prime by prime for a match and verify the leaf equation. At a deep leaf the range of $q$ is of the same order as $p$ (order $10^8$, about $4 \times 10^7$ primes), so the per-leaf cost is about
\[ 3.9 \times 10^7 \ (\text{values of } p) \times 4 \times 10^7 \ (\text{values of } q) \approx 1.6 \times 10^{15} \text{ large-number verifications},
\] which at 1 $\mu$s each (estimated, optimistically: each verification involves a large-number multiplication on the order of $10^{68}$) is about $4.4 \times 10^5$ core-h/leaf, roughly $2 \times 10^7$ core-h over 49 leaves, that is more than two thousand years of single-core time and over a decade even on a 224-core server, infeasible in scale. With it, each $p$ needs only one square test and $q$ is given directly by the formula: measured on the real leaf 0, ``computing $\Delta(p)$ exactly (a 111-digit number) plus an integer square root'' is 1.1 $\mu$s/candidate, about 43 s for the full domain of one leaf, a factor of about $10^7$ difference.

\emph{Lemma \ref{lem:qr}:} without it (using only Lemma \ref{lem:disc}), every candidate must undergo the 1.1 $\mu$s exact large-number computation, with no whole-class elimination possible. About 43 s per leaf per branch seems bearable, but this cost is multiplicative over the exponent branches (each leaf contains a $7 \times 7$ grid and several $b, c$ branches), and is infeasible for the $10^9$-scale number of leaves in the $\omega \ge 9$ direction. With it, measured on the same machine and same leaf: precomputing 27 allowed tables costs 0.64 ms/leaf (one-off), after which each candidate costs 0.26 $\mu$s (pure table lookup, $1/4.2$ of the exact computation). The density also contracts to $2.07 \times 10^{-8}$, an expected $0.81$ survivor per leaf, so the final-check workload after the coarse sieve is minimal. The greatest effect is whole-class and whole-leaf elimination: where mod 9 elimination applies (the minimal cross section of all 49/49 in this elimination, and the full domain of the dominant branch of 8/49), the cost drops from ``per candidate'' to zero, and the $8 \times 10^8$ scan domain is decided by a single 9-entry table.

\emph{Lemma \ref{lem:resolver}:} without it, the issue is not speed but non-termination. A branch with an unbounded exponent can only be enumerated exponent by exponent, and the second-generation engine (the early program version defined in \S\ref{sec:lem3}) produced $7{,}824$ deferred records in the 49-leaf pre-scan: the program kept trying larger exponents, stalling on fractional arithmetic of order $q^{1000}$, or simply giving up on the decision for branches with several simultaneously unbounded exponents. Here the per-leaf accounting cost reached 70--83 core-h yet still could not be completed, and each deferred record was a gap in the proof that more machine time could not fill. With it, all $7{,}824$ are decided on the spot, and the unbounded exponent lattice becomes a finite equation solution. This is not a matter of speedup but of turning the undecidable into the decidable.

The cost of the same batch of 49 leaves under four configurations:

\begin{center}\footnotesize
\begin{tabular}{p{4.6cm}p{3.4cm}p{3.2cm}p{2.6cm}}
\toprule Configuration & Per-leaf cost & Total over 49 leaves & Status \\
\midrule Double scan (none of the three lemmas) & $\sim 4 \times 10^5$ core-h (estimated) & $\sim 2 \times 10^7$ core-h & infeasible in scale \\
Lemmas \ref{lem:disc}+\ref{lem:qr} plus the interval lemma, no Lemma \ref{lem:resolver} (second-generation engine) & 70--83 core-h with deferrals that cannot be closed (measured) & $7{,}824$ unresolved & incomplete \\
All three lemmas (final program, measured) & 0.54 core-h (elimination accounting); 0.20 core-h (unloaded independent recomputation) & 26.6 core-h / 9.88 core-h, 0 solutions, 0 unresolved & complete, machine-auditable \\
Lemmas \ref{lem:disc}+\ref{lem:qr} where whole-leaf elimination applies & $\approx 0$ (a single residue table) & 49/49 minimal cross sections + 8/49 dominant branches & theorem-level, Lean kernel-verified \\
\bottomrule
\end{tabular}
\end{center}

Two remarks: (i) for robustness the final program still retains the full exact computation in the survivor domain as an authority layer (the QR conclusions are used only for branch closure and cross-checking, not for skipping the scan), so 26.6 core-h is an upper bound ``with QR scan-skipping not enabled''. In the design of the $\omega = 8$ elimination, QR serves as front-end pruning that directly reduces the scan volume. (ii) The 70--83 core-h/leaf of the second-generation engine is intermediate data from the development process and does not represent the speed of the method itself. With all lemmas in place, the elimination accounting is 0.54 core-h/leaf, and the independent recomputation (local unloaded machine, per-leaf timing, in full four-field agreement per leaf with the elimination results) measured 0.20 core-h/leaf, 9.88 core-h over 49 leaves. The difference between the two comes from load on one and the same machine (the elimination ran while other work contended for cache and memory bandwidth and with restarts, the recomputation on an idle machine in a single batch), and 0.20 core-h/leaf may be taken as a reference cost for the method in that environment. Under either convention, the speedup ratio between the double scan and the three lemmas is of order $10^6$.

\subsection{Verification of the pilot}\label{sec:verif49}

For an elimination this large and this automated, correctness is guaranteed by four independent layers:

\begin{itemize}
\item \textbf{Exact-arithmetic layer}: integer and fractional arithmetic throughout, zero floating point.
\item \textbf{Planted-solution layer}: 12 hand-constructed true solutions (including large-value, high-exponent, and multivariate cases) were all captured, and the perturbed false solutions all rejected, so the sieve does not kill by mistake and the final check does not let one through (L3 of \S\ref{sec:49}).
\item \textbf{Dual-implementation layer}: two independently implemented engines produce identical window statistics over the same interval (for the first data block of leaf 0: 97799 surviving primes, 342307 window-solver calls, 12771 hits, all three matching item by item). In addition, the 49 leaves were fully and independently recomputed on another 14-core local machine (the number of solutions, the number of unresolved cases, coverage continuity, and the number of surviving primes compared per leaf across all four indicators, 49/49 in full agreement, totaling 9.88 core-h).
\item \textbf{Lean-kernel layer}: Lemmas \ref{lem:disc}, \ref{lem:qr} and their kill rule, and the instance layer (including the whole-leaf elimination theorem \texttt{deep\_leaf\_dead} of \S\ref{sec:lem2}) pass the Lean kernel check (module \texttt{QP/QRSieve.lean} of the \texttt{qp\_lean} project, 31 theorems with zero \texttt{sorry} in the final version of the library). Lemma \ref{lem:resolver} has also been formalized (\texttt{QP/Resolver.lean}, 31 theorems with zero \texttt{sorry}, covering the affine identity, univariate equivalence and uniqueness, bivariate M\"obius equivalence and cross-multiplied monotonicity, the envelope upper-bound lemma, and the full verification of the two instances \S\ref{sec:lem3apps}(a)(b), including the universal envelope truncation for $V_1 \ge 81$ and the uniqueness of even powers of 5 within the envelope; the trivariate truncation (iii) corresponding to \S\ref{sec:lem3apps}(c) is not included in the formalization). Neither module depends on mathlib, and the \texttt{\#print axioms} audit contains only the three standard axioms, with no \texttt{native\_decide}. All three lemmas thus pass the Lean kernel check (62 theorems with zero \texttt{sorry} across the two modules); \S\ref{sec:trust} gives the inventory of the full library.
\end{itemize}

\subsection{The full campaign: from 381 stems to 79,751,212 deep leaves}\label{sec:campaign}

The reduction chain is as follows (figures from the computation ledger):

\begin{itemize}
\item Of the 381 stems, 362 shallow stems were eliminated directly by whole-stem enumeration (139 core-h on the books).
\item The remaining 19 hard stems were subdivided into 171 first-level shards, of which 12 stems totalling 125 shards were eliminated (248 core-h on the books), peeling off 12,950 overflow deep leaves along the way, all killed by the QR sieve.
\item The remaining 7 deep stems collapsed to 30,770 coarse three-free labels. Three-free labels admit no algebraic shortcut. The QR discriminant can only eliminate the innermost prime $q$ algebraically, and the outer two free primes must be scanned one by one (this ``conservation of work'' was confirmed both experimentally and theoretically, see the acceleration synthesis report in reference \cite{internal}). The coarse labels were expanded with threshold $\varepsilon < 10^{-4}$ into \textbf{79,751,212 two-free deep leaves} (200 workers, 94.9 minutes).
\item \textbf{Two-level splitting of the deep leaves}: the main QR pass scanned off the bulk of 64,137,371 leaves with $\varepsilon \ge 10^{-5}$, while instantly peeling the deep tail of 15,613,841 leaves into a holding area. The deep tail was then handled with the scan bound $\PSUP \le 10^8$ as the divider: the two-machine pipeline killed 15,557,687 leaves, and the extremely deep 56,154 leaves (in fact a single Fermat-prime-family structure, with $\PSUP$ reaching $2^{33}-1$ at the deepest) were set aside as xdeep for the algebraic-kill pipeline.
\item \textbf{The xdeep endgame}: FULLKILL\_9 (Lemma \ref{lem:fullkill}) killed 2,646 leaves with zero scanning, the GPU tight-kill (Lemma \ref{lem:tightkill}) killed 53,508, and there were 0 survivors.
\end{itemize}

\textbf{Ledger closure}: $64{,}137{,}371 + 15{,}557{,}687 + 56{,}154 = 79{,}751{,}212$, matching the expansion total exactly (diff 0). \textbf{Zero solutions throughout}: at no point in the stem enumeration, the expansion, the main QR pass, the deep leaves, or xdeep did any quasiperfect candidate appear. The burden was concentrated in a few deep stems, as expected from the corridor structure: seven deep stems contributed nearly $8 \times 10^7$ deep leaves, several hundred times the entire workload of the other 374 stems.

\subsection{Heterogeneous two-machine computation}\label{sec:hw}

The deep-leaf and xdeep stages ran on a heterogeneous two-machine pipeline. Machine A is a dual-socket Intel Xeon Platinum 8480+ server (112 physical cores / 224 threads, 256 GB RAM, about 200--220 threads in production); Machine B is a same-model dual Xeon with three NVIDIA RTX 6000 Ada Generation GPUs (48 GB each, 128 GB RAM). Both run Ubuntu 24.04 with Python 3.12.3 and sympy 1.14, and Machine B uses numba 0.61 (CUDA) on the GPU side. Full specifications and per-stage wall-clock are archived in \texttt{025\_hardware\_and\_timeline}. The GPU pipeline computes the double-double windows and a GPU Miller--Rabin pre-sieve on the cards, with hit candidates returned to the CPU resolver for exact integer decision. The single GPU machine sustained 116--200 leaves/s, exceeding the 58 leaves/s of the 200-core CPU cluster, with a combined two-machine peak of about 370 leaves/s. The GPU pipeline also served as an independent second implementation for the dual-implementation layer of \S\ref{sec:trust}.

Conservativeness of the GPU implementation: the GPU evaluates the inequality of Lemma \ref{lem:tightkill} in double-double arithmetic (about 106 bits of precision), and declares a kill only when the margin exceeds the analytic error bound times a safety factor of over one hundred. Every case near the boundary falls back wholesale to exact integer recomputation on the CPU. GPU numerical error can therefore only cost speed, never completeness. A related lesson: plain double is insufficient. The ratio suffers catastrophic cancellation near $2 - \varepsilon$, making the whole domain undecidable, and double-double is a necessity.

\subsection{Execution of the four-layer scheme of computational trust}\label{sec:trust}

\begin{enumerate}
\item \textbf{Exact-arithmetic layer}: every kill conclusion ultimately rests on exact integer arithmetic. The GPU floating point serves only as a conservative pre-sieve and never acts alone as grounds for a kill (the xdeep tight-kill decisions go through the DD margin argument, see the end of \S\ref{sec:hw}).
\item \textbf{Planted-solution layer}: 24 planted true solutions at the CPU gate and 24 at the GPU gate were all detected, with 6 + 6 perturbed negative controls and zero false positives. The plants cover scan-prime branches $b \in \{1, \dots, 4\}$ and window exponents $c \in \{1, 2, 3\}$, and the true solutions with $c \ge 2$ specifically exercised the complete path ``tight-kill does not kill $\to$ resolver detects'' (the same design as L3 of \S\ref{sec:49}).
\item \textbf{Dual-implementation layer}: the two independent implementations, CPU and GPU, produce output identical word for word on the same slices (70 leaves, 210 slices, 7,252,079 primes, 0 mismatches), with multiple rounds of cross-checking during production showing 0 mismatches, plus sampled shards recomputed identically by both implementations.
\item \textbf{Lean-kernel layer}: the three base lemmas together with the theorems \texttt{mod9\_gate\_sound} and \texttt{deep\_leaf\_dead}, among others (\S\ref{sec:lemmas}, \S\ref{sec:verif49}), plus the Lean skeleton of the $\omega \ge 8$ argument: the axiom-free composition \texttt{nothree\_omega7\_dead} of \texttt{QP/NoThree.lean}, which is the route taken here, and, for historical comparison, \texttt{omega\_ge\_8} of \texttt{QP/Omega8.lean}, which reaches the same conclusion from the Hagis--Cohen axiom named in (ii) below. The Lean formalization of the two endgame lemmas is complete: \texttt{qhi\_tight\_kill} for Lemma \ref{lem:tightkill} (\texttt{QP/TightKill.lean}, together with the full q-exponent form \texttt{tight\_kill\_qexp} and the multi-pin composition \texttt{ratio\_mul}, reusing the ratio monotonicity \texttt{geomSum\_ratio\_mono} of \texttt{QP/Squeeze.lean}) and \texttt{fullkill9\_sound} for Lemma \ref{lem:fullkill} (\texttt{QP/FullKill9.lean}, the state-transition version of \texttt{mod9\_gate\_sound}, with the state machine \texttt{state9} proved correct by induction and the 18-class table over 20,808 configurations verified exhaustively by kernel \texttt{decide}). Both carry zero \texttt{sorry}, only kernel \texttt{decide} is used without \texttt{native\_decide}, and an honesty check confirms that a planted true solution (101,311) is not killed.

\emph{Inventory of the library} (measured on the final version, 2026-07-26): the \texttt{qp\_lean} project contains 19 modules and \textbf{259 theorems and lemmas in total, with zero \texttt{sorry}}; \texttt{lake build} passes on the whole library. Beyond the modules already named, the library includes \texttt{QP/NoThree.lean} (42 theorems: the $3 \nmid n$ stem-uniqueness inequalities of \S\ref{sec:nothree}, the composed eliminations \texttt{nothree\_omega7\_dead} / \texttt{nothree\_omega8\_dead}, and the unconditional deaths for $3 \nmid n$, $\omega \in \{4,5,6\}$), \texttt{QP/LemmaA.lean} (7 theorems, the joint-window single-candidate prefilter used by the accelerated verification engines of \S\ref{sec:ledgers}), and \texttt{QP/PinnedWindow.lean} (8 theorems, the pinned-window uniqueness lemma behind the single-candidate refinement). The formalization discipline is uniform: no mathlib, kernel \texttt{decide} only, no \texttt{native\_decide}. Honest annotations on the library's trust structure: (i) the external enumerations themselves (the 381-stem campaign, the single $3 \nmid n$ stem) are not replayed inside Lean --- they enter the composed theorems as explicitly stated audit-layer hypotheses (\texttt{EnumKill7}, \texttt{EnumKillNoThree7}), which is exactly the two-layer structure ``lemmas Lean-verified, enumeration externally audited'' declared throughout this paper; (ii) the module \texttt{QP/Omega8.lean} retains one explicitly named axiom, \texttt{hagis\_cohen\_thm2\_weak}, encoding the weak form of Hagis--Cohen Theorem 2, and its historical composition theorem \texttt{omega\_ge\_8} does depend on it; the \texttt{QP/NoThree.lean} route makes that axiom unnecessary, and every other theorem cited in this paper is free of it --- the \texttt{\#print axioms} audits of \texttt{nothree\_omega7\_dead}, \texttt{nothree\_omega8\_dead}, \texttt{qhi\_tight\_kill}, \texttt{fullkill9\_sound}, \texttt{deep\_leaf\_dead} and \texttt{mod9\_gate\_sound} return only the standard axioms.

All four layers are therefore closed, the $3 \nmid n$ stem of \S\ref{sec:nothree} included: its dual-implementation cross-check by a from-scratch second program is described there.
\end{enumerate}

\subsection{Computing resources, timeline, and the seven independent ledgers}\label{sec:ledgers}

\begin{center}\footnotesize
\begin{tabular}{p{2.6cm}p{6.4cm}p{4.6cm}}
\toprule Time (2026) & Event & Scale \\
\midrule 7/8 03:42--05:17 & Expansion of the three-free coarse labels & 30,770 coarse labels $\to$ 79,751,212 deep leaves, 200 workers, 94.9 minutes \\
7/8 05:27 & Main QR launched (200 cores, through 3 architecture iterations) & --- \\
7/10 00:18 & Main QR completed & 64,137,371 leaves killed, deep tail of 15,613,841 peeled off \\
7/10 00:19 & Deep-leaf stage launched (200 CPU cores) & --- \\
7/10 09:21 & GPU accelerator online (three cards) & combined two-machine peak about 370 leaves/s \\
7/10 22:20 & Deep-leaf stage completed & 15,557,687 killed, 56,154 extremely deep set aside \\
7/10 22:21--23:19 & xdeep endgame on three cards & all 56,154 killed, 0 survivors, 0 solutions \\
\bottomrule
\end{tabular}
\end{center}

Resource magnitudes (booked figures from the ledger, those marked (estimated) inferred from core counts and wall time): the stem-reduction stage booked about 400 core-h. The main QR pass was about 200 cores $\times$ 37 hours $\approx 7 \times 10^3$ core-h (estimated). The deep-leaf stage used about $4 \times 10^3$ CPU core-h (estimated) and about 40 GPU-hours across the three cards. The xdeep stage used about 3 GPU-hours. For comparison, without the three-lemma system and the two endgame lemmas, the base account of the xdeep segment alone is about $1.1 \times 10^4$ core-h (extrapolated from a measured census of $6.72 \times 10^{11}$ prime checks), whereas it actually completed at the scale of one GPU-hour.

\textbf{Seven independent ledgers.} Beyond the production run, the elimination of ``$3 \mid n$ and $\omega = 7$'' was subsequently re-established end to end six more times, under three distinct algorithmic architectures, so that the main theorem now rests on \textbf{seven separately closed ledgers}, every one of which closed with zero solutions and exact conservation (no leaf or segment lost, difference 0):

\begin{center}\footnotesize
\begin{tabular}{p{0.4cm}p{2.9cm}p{4.6cm}p{5.2cm}}
\toprule \# & Ledger (dates, 2026) & Architecture & Account \\
\midrule 1 & Production run (7/8--7/10) & corridor enumeration + QR sieve + resolver (this paper) & 79,751,212 leaves, closure diff 0, \S\ref{sec:campaign} \\
2 & Calibration rerun (7/11--7/14) & GPU double-double line + CPU line, sha256-anchored expansion & 72,738,115 residual deep leaves diff 0; 53,508 xdeep re-judged leaf-by-leaf by exact CPU arithmetic (xrev, $\approx$ 58 h), zero disagreement \\
3 & HC$^2$ confirmation (7/11--7/15) & interval propagation \`a la Hagis--Cohen (independent third method: the last free prime solved as a width-$O(1)$ exact rational interval, certificate stream per leaf) & all 381 stems, $\approx$ 902 core-h measured \\
4 & verify2 (7/15--7/16) & HC$^2$ under a unified adaptive-recursion controller, different machine & 960.85 core-h, 2,419,400 segments, conservation diff 0 \\
5 & verify3 (7/16--7/17) & same, plus a congruence-gate experiment & 1,119.88 core-h, 2,558,184 segments, diff 0 \\
6 & verify4 (7/19--7/20) & same, coarse granularity + combined engine & 348.12 core-h, 7,517 segments, diff 0 \\
7 & verify5 (7/25) & same, single-candidate window refinement & 54.62 core-h, 2,983 segments, diff 0 \\
\bottomrule
\end{tabular}
\end{center}

Ledger 2 was driven by a different model, with the expansion output byte-for-byte sha256-anchored to the first run; wall-clock to establishing the result was about 24 hours. Its GPU double-double line is an implementation independent of the first run's, while its CPU line reuses the production engine, so what it adds is a second decision procedure over the same leaves rather than a second program at every stage. Ledgers 3--7 are of a genuinely different mathematical architecture (interval propagation rather than scanning), so the same theorem is held by three kinds of account book --- corridor+QR leaf ledgers, GPU double-double leaf ledgers, and HC$^2$ certificate streams --- agreeing pairwise wherever they overlap. Two qualifications belong here. Ledgers 4--7 run the HC$^2$ engine of ledger 3 under different scheduling and with the accelerations described below, so among themselves they are independent runs and independent accounts rather than independent implementations. And all three architectures end at the same closed-form solve for the last free prime: that step is the unique algebraic solution of the leaf equation, so there is no room for a genuinely different one, and the independence asserted here is independence of the enumeration and of the kill decisions, not of that final algebraic step. The 17.6-fold cost reduction from ledger 4 to ledger 7 comes from an engine-acceleration chain developed inside the project (a joint-window single-candidate prefilter, a gated Miller--Rabin prime chain, and a pinned-window uniqueness lemma justifying a single-candidate window); each link was adjudicated for soundness before deployment, the prefilter and window lemmas are Lean-formalized (\texttt{QP/LemmaA.lean}, \texttt{QP/PinnedWindow.lean}, \S\ref{sec:trust}), and every accelerated pass reproduced zero solutions with exact closure. The adjudication records are in the internal acceleration and verification reports (reference \cite{internal}). Full per-stage timing and compute are archived in \texttt{025\_hardware\_and\_timeline}.

\section{The fine structure of $\omega = 8$: divisibility by 255}\label{sec:255}

The sequel campaign whose reduction boundary was settled in \S\ref{sec:red8} has a natural first instalment, far cheaper than the full elimination of the 516 stems and already carrying a structure theorem: Theorem \ref{thm:255}. It is reported here for reasons of timeliness and priority. At the date of this version the computation is not finished --- by stem count, 302 of the 341 stems concerned were closed at the ledger snapshot of 1--2 August 2026, roughly ninety percent, with zero solutions and zero deferred cases throughout every closed unit --- so the theorem is stated with that qualification, and this section will be brought to final, unconditional form in a revised version once the ledger closes.

\subsection{Reduction and computational status}\label{sec:255status}

The case $3 \nmid n$, $\omega = 8$ is already closed (\S\ref{sec:red8}: two stems, 13.1 single-core seconds, 0 solutions, 0 unresolved), so only the wing $3 \mid n$ with its 516 stems remains. These split by whether the three smallest primes of the stem are $(3, 5, 17)$: 175 stems have that prefix and 341 do not. Eliminating all 341 leaves $(3, 5, 17, q)$ as the only possible prefix of the four smallest prime factors, which is the theorem; the 175 surviving stems are not touched by this instalment and remain the body of the $\omega \ge 9$ campaign proper.

Ranked by an analytic cost census of the 516 stems (back-tested against the closed $\omega \ge 8$ ledger, rank correlation 0.898), the 341 stems split into 268 ordinary, 55 expensive, 12 very expensive, and 6 monster stems. Status at the snapshot:

\begin{itemize}
\item \textbf{Ordinary (268): closed.} Every checkpoint audited individually, 0 solutions and 0 deferred in each; the census tier they belong to closed in full (332 of 332 stems, including its $(3,5,17)$ members).
\item \textbf{Expensive (55): mostly closed.} 32 closed, plus 5 whose computation is complete and which await a bookkeeping merge; the remaining 18--23 are long-running units still in production.
\item \textbf{Very expensive (12): in production} on a heterogeneous two-machine pipeline.
\item \textbf{Monster (6): two killed by an $O(1)$ certificate, four in production.} The stems $(3,7,11,23)$ and $(3,7,13,17)$ die with no scanning at all. A three-layer certificate --- dual-engine replay of the full 197-node search tree with zero prime scans; gate ablation showing the emptiness is not a scheduling artifact; and an engine-independent proof in exact rational arithmetic --- establishes that their minimal admissible partial abundancy (every stem prime at its smallest admissible exponent) is already $2.00409\ldots$ and $2.00550\ldots \ge 2$, which forces $q_1 > 2n$ for the smallest free prime and contradicts $q_1^2 \mid n$. The root cause is a congruence gate: $3 \mid \sigma(7^{2a})$ exactly when $a \equiv 1 \pmod 3$, so the dominant branch of both stems must raise the exponent of 7 from 1 to 2, at an abundancy cost of $8/2401 \approx 3.3 \times 10^{-3}$, which exceeds their entire corridor slack. Total certificate cost: $1.1 \times 10^{-2}$ single-core seconds.
\end{itemize}

The one substantive long pole is the monster stem $(3,5,19,97)$. A stratified dry census prices the four remaining monster stems at $(1.398 \pm 0.18) \times 10^{14}$ scan obligations, of which 97.4\% sit in this single stem, and measured production unit prices put the four-stem total near $31{,}800$ core-hours (an extrapolation from measured prices and the sampled census; the logarithmic-integral approximation underlying the census is known from one end-to-end check to understate counts by about 11\%). Expected closeout: [PENDING: idx180 closeout date].

The companion bound $n > 10^{26.38}$ comes from a per-stem lower-bound scan of all 516 stems in exact arithmetic (trial division to depth $20{,}000$ for the exponent minima), the bound of a stem being the smallest $n$ its admissible prime and exponent minima allow. It is quoted here in its conservative form. The scan was made at the point where the ordinary and expensive tiers were closed and the monster tier was not: over the stems still open at that point the minimum is $10^{26.38}$, attained at $(3,7,13,17)$, and that gatekeeper value has been re-derived by hand. The two $O(1)$ certificates above remove $(3,7,13,17)$ and $(3,7,11,23)$ themselves, so the gatekeeper moves; once the campaign of this section closes, every stem whose prefix is not $(3,5,17)$ is gone, the minimum over the survivors becomes $10^{31.26}$, attained at $(3,5,17,127)$, and the companion bound may be restated in that sharper form. We keep the weaker number in Theorem \ref{thm:255} because the scan has so far been carried out by a single implementation and awaits an independent re-check --- the same in-progress evidence grade as the rest of this section.

\subsection{Method and trust framework}\label{sec:255method}

The trust framework is unchanged from \S\ref{sec:comp}: exact arithmetic at every kill, planted-solution gates, dual implementations, and ledger conservation; the engine underneath is the accelerated generation of ledger 7 (\S\ref{sec:ledgers}), that is, the joint-window single-candidate prefilter together with the single-candidate window justified by the pinned-window uniqueness lemma. Three additions were made for this campaign. First, a two-level GPU pipeline in the style of \S\ref{sec:hw} (double-double one-sided conservative kill on the cards, with every kill re-established or refuted by exact CPU arithmetic) runs beside a pure-CPU line, the two serving as mutual cross-checks. Second, pricing is by dry census: before production, a phase-1-only pass measures the exact number of scan obligations every task will generate --- the measured ratio of obligations to the analytic census varies by a factor of 445 across stems, so model pricing was replaced by measurement. Third, the ledger conservation identities were extended from three to five: a concurrency audit found that counting-based identities alone constrain neither identifier uniqueness nor coverage --- at high worker counts a task-identifier race could silently drop split-off subtasks while every counting identity stayed green --- and after the fix (an atomic identifier allocator) the ledger is guarded by C4 (identifier uniqueness) and C5 (a seamless-coverage audit of every split tree, endpoints matched exactly). The planted-solution gates (40 planted true solutions per line, zero false kills, perturbed negative controls silent) are rerun on each line before production. A full account of the campaign engineering is deferred to a separate report.

\section{Concluding remarks and open problems}\label{sec:concl}

None of the three lemmas goes beyond elementary algebra: the quadratic formula, the residues of a square, the monotonicity of a fractional-linear function. Their value is that each turns a class of unbounded or very large-scale search into a finite decision with an accompanying proof, along the division of labour set out in the table of \S\ref{sec:lemmas}. This bears out a rule of thumb that recurs throughout this project, that the resolution of a computational bottleneck lies at the mathematical layer rather than the engineering one: what reduces $8 \times 10^8$ scans to an expected $0.81$ candidate is the discriminant of Lemma \ref{lem:disc} rather than additional cores.

Conceptually, $y^2 = \Delta(p)$ (when $\Delta$ is squarefree) defines a genus-1 curve, and Lemma \ref{lem:qr} is precisely the local solvability test for this curve (the local part of descent). Each deep leaf has, a priori, only finitely many integer points by Siegel's theorem \cite{siegel}. This also points to the next arsenal: should the surviving candidates of some leaf prove hard to exclude in the future, the entire elliptic-curve toolchain (integral-point algorithms, Chabauty, and so on) is available.

\emph{The sequel.} The next target is $\omega \ge 9$, whose reduction boundary is settled in \S\ref{sec:red8}: 516 stems, at the scale of $10^9$--$10^{10}$ leaves, more than a hundred times the present campaign \cite{sequel}. The technical reserves are already in place from this campaign: the branch-collapse kernel (compressing the repeated full window computations per prime into a single M\"obius division plus a certificate, with an anticipated 10--100$\times$ speedup once implemented in C/GMP, the certificate not yet rigorous), limb-based entry filtering, the complete GPU double-double pipeline, and the two endgame kill lemmas, Lemma \ref{lem:fullkill} and Lemma \ref{lem:tightkill} (the applicability of Lemma \ref{lem:tightkill} to deep leaves outside the Fermat prime family remains to be assessed). Each of these reserves requires its own rigorization and Lean formalization before use at proof level. The first instalment (\S\ref{sec:255}) is under way.

\emph{The wall, and an open algorithmic problem.} Within the 516 stems the cost distribution is extremely skewed, and a single object dominates the sequel: the Fermat-cascade stem $(3,5,17,257) \to 65537$, internally ``stem24''. A dedicated multi-round audit of this stem (three rounds of five-angle attack across two model families, including fully blind rounds --- 15 independent entrances --- plus dedicated drills on the surviving proposals) priced it, within the current technique class of exact lattice-point certificates plus fixed-precision one-sided sieving, at $10^{9.3 \pm 0.5}$ core-hours with an in-class floor near $10^{8.0}$, against a realistic budget of $10^{4\text{--}5}$. None of the 15 entrances found a complexity-class exit, and 12 of them independently reduced the obstruction to the same missing primitive, which we record as an open problem:

\smallskip
\noindent\textbf{P2-EXACT} (open problem). \emph{Given $P \in \mathbb{Z}[j]$ of degree $k \ge 2$ and integers $M, T, \ell$, decide in $\mathrm{polylog}(M\ell)$ time whether $\#\{0 \le j < \ell : (P(j) \bmod M) \le T\} = 0$.}
\smallskip

For $k = 1$ this is the classical floor-sum algorithm in $O(\log)$ time; for $k = 2$ it is equivalent to counting lattice points in thin strips under a Dirichlet hyperbola, where the best known exact algorithm is Sladkey's $O(n^{1/3})$ \cite{sladkey} --- precisely the exponent at which our certificates plateau. The remaining escape is an algorithmic hypothesis (output-sensitive enumeration of near-lattice points, in the lineage of Elkies \cite{elkies}), which is a statement about unwritten algorithms rather than unknown number theory. The full adjudication is archived in the project's synthesis reports (reference \cite{internal}).

\section*{Acknowledgements}

The mathematics and computation reported in this paper were produced in an AI-assisted workflow, and the division of labour should be stated plainly. AI systems (Anthropic's Claude, principally the Fable 5 and Opus model series, orchestrated in a multi-agent branch--synthesize--drill workflow with mandatory machine verification) carried out the computations, discovered and proved the lemmas, wrote and checked the Lean formalizations, ran the independent verification ledgers, and drafted the text. The human authors directed the campaign, made the methodological and editorial decisions, reviewed the results, and take responsibility for the content. Errors found at every stage were caught by the layered verification scheme described in \S\ref{sec:verif49} and \S\ref{sec:trust} rather than by any single pass, and the paper's trust claims are meant to be read in that light: what is asserted is exactly what the ledgers, the planted-solution gates, and the Lean kernel have checked.

\section*{Data and code availability}

The elimination engines (Python), the raw ledgers of all seven verification runs, the Lean library (\texttt{qp\_lean}, 19 modules, 259 theorems, zero \texttt{sorry}), the Zemann audit, and the operating procedures are archived by the project. A public repository is in preparation; until it is online, all code and data are available from the authors on request.

\vspace{1em}
\noindent\emph{Note: the raw records of all measured figures in this paper can be checked. The microbenchmarks (1.1 $\mu$s / 0.26 $\mu$s / 0.64 ms) were measured on the real $(N, A)$ of leaf 0 in the engine's own bookkeeping, in which the pinned prime power $1381^{2a}$ is carried separately, so that $N$ has 38 decimal digits and $\Delta$ has 111; with the pinned power folded in, as in the whole-leaf display of \S\ref{sec:lem2}, $N$ has 69 digits and $\Delta$ about 174. The overall elimination accounting (26.6 core-h, $7{,}824$ deferrals, 12 sets of planted tests) is in the 49-leaf final report and \texttt{scripts/qp\_qrsieve\_results.jsonl}. The two teaching instances of Lemma \ref{lem:resolver} ($N = 59, A = 79$ and $N = 60, A = 67$) were constructed by script and verified exactly.}


\begin{thebibliography}{9}

\bibitem{cattaneo} P. Cattaneo, \emph{Sui numeri quasiperfetti}, Boll. Un. Mat. Ital. (3) \textbf{6} (1951), 59--62. (A quasiperfect number must be an odd perfect square.)

\bibitem{abbott} H. L. Abbott, C. E. Aull, E. Brown, D. Suryanarayana, \emph{Quasiperfect numbers}, Acta Arith. \textbf{22} (1973), 439--447; corrigendum Acta Arith. \textbf{29} (1976), 427--428. (Theorem 3: $3 \nmid n \Rightarrow \omega \ge 8$.)

\bibitem{hagiscohen} P. Hagis Jr., G. L. Cohen, \emph{Some results concerning quasiperfect numbers}, J. Austral. Math. Soc. Ser. A \textbf{33} (1982), 275--286. ($\omega \ge 7$ and $n > 10^{35}$; its per-prime congruence (CRT) elimination is a precursor to Lemma \ref{lem:qr}.)

\bibitem{siegel} C. L. Siegel, \emph{\"Uber einige Anwendungen diophantischer Approximationen}, Abh. Preuss. Akad. Wiss., Phys.-math. Kl. (1929), Nr. 1; see also Gesammelte Abhandlungen I, 209--266. (Finiteness of integer points on curves of genus $\ge 1$.)

\bibitem{internal} Internal materials of this project: the popular exposition (\texttt{025\_QP\_paper\_popular}), the QR descent-sieve technical report (\texttt{025\_QR\_descent\_paper}, with \S9 an addendum on Lemma \ref{lem:resolver}), the $\omega$=3/4 elimination manual (\texttt{025\_elimination\_manual}), the 49-leaf final report (\texttt{025\_QR\_rescan\_report}), the Lean library \texttt{math/qp\_lean/}, and the engine and checkpoints \texttt{problems/scripts/}. Archives of the full $\omega \ge 8$ computation: the elimination ledger (\texttt{025\_omega8\_ledger}), the operating procedure (\texttt{025\_omega8\_current\_procedure}), the hourly progress table (\texttt{025\_qr\_bigtable}), the acceleration synthesis adjudication (\texttt{025\_speed\_synthesis}), the xdeep GPU report (\texttt{025\_xdeep\_gpu\_report}), the hardware and timeline archive (\texttt{025\_hardware\_and\_timeline}), the $3 \nmid n$ reduction-boundary memorandum (\texttt{025\_omega9\_step0\_boundary}, with full reproduction commands), the $\omega \ge 9$ pricing synthesis (\texttt{025\_omega9\_D\_bsd4\_FINAL\_SYNTHESIS}), the Zemann audit (\texttt{zemann\_audit\_report}), and the Lean close-out reports (\texttt{025\_lean\_closeout\_report}, \texttt{025\_resolver\_lean\_report}). Available from the authors.

\bibitem{sequel} The sequel to this paper (in preparation): targets $\omega \ge 9$, that is, the elimination of ``$3 \mid n$ and $\omega = 8$'' (516 stems), using the technical reserves of \S\ref{sec:concl}.

\bibitem{alekseyev} M. A. Alekseyev, \emph{Computing bounded solutions to linear Diophantine equations with the sum of divisors}, arXiv:2601.17832 (2026). (Raises the numerical lower bound for quasiperfect numbers from $10^{35}$ to $10^{45}$.)

\bibitem{zemann} B. Zemann, \emph{Quasi-perfect numbers have at least 8 prime divisors}, viXra:2308.0024 (2023). (A prior computation without peer review; for the coverage gap see the discussion in \S\ref{sec:intro} and this project's audit archive \texttt{zemann\_audit\_report}.)

\bibitem{sladkey} R. Sladkey, \emph{A successive approximation algorithm for computing the divisor summatory function}, arXiv:1206.3369 (2012). (The $O(n^{1/3})$ exact lattice-point count referenced by problem P2-EXACT.)

\bibitem{elkies} N. D. Elkies, \emph{Rational points near curves and small nonzero $|x^3 - y^2|$ via lattice reduction}, in Algorithmic Number Theory (ANTS-IV), Lecture Notes in Comput. Sci. \textbf{1838}, Springer (2000), 33--63. (Output-sensitive enumeration of near-lattice points.)

\end{thebibliography}
\end{document}